\documentclass[10pt]{amsart}
\usepackage{amsmath, amsthm, amsfonts, amssymb}
\usepackage{enumerate}
\usepackage[bookmarks=false]{hyperref}

\usepackage{xcolor}

\usepackage{verbatim}

\usepackage{comment}

\newtheorem{thm}{Theorem}[section]
\newtheorem{prop}[thm]{Proposition}
\newtheorem{col}[thm]{Corollary}

\theoremstyle{definition}
\newtheorem{defn}[thm]{Definition}
\newtheorem{defn/lem}[thm]{Definition/Lemma}
\newtheorem{rmk}[thm]{Remark}

\newtheorem{examples}[thm]{Examples}

\newcommand{\C}{\ensuremath{\mathbb{C}}}

\newcommand{\M}{\ensuremath{\mathbb{M}}}
\newcommand{\N}{\ensuremath{\mathbb{N}}}
\newcommand{\Z}{\ensuremath{\mathbb{Z}}}
\newcommand{\F}{\ensuremath{\mathbb{F}}}
\newcommand{\cU}{\ensuremath{\mathcal{U}}}

\newcommand{\sk}{^{(k)}}
\newcommand{\si}{^{(\infty)}}
\newcommand{\Sym}{\mathrm{Sym}}

\newcommand{\GL}{\text{GL}}
\newcommand{\Aut}{\text{Aut}}

\begin{document}

\title[Strongly converging unitary representations for extensions by exact groups]{Strongly converging unitary representations for extensions by exact groups}

\author[D. Gao]{David Gao}
\address{Department of Mathematics, UC Berkeley, Evans Hall, CA 94705, USA}\email{d-gao@berkeley.edu}\urladdr{https://sites.google.com/view/david-gao}

\author[S. Kunnawalkam Elayavalli]{Srivatsav Kunnawalkam Elayavalli}
\address{Department of Mathematics, UC Berkeley, Evans Hall, CA 94705, USA}\email{srivatsav.k.e@berkeley.edu}
\urladdr{https://sites.google.com/view/srivatsavke/home}

\author[M. Mj]{Mahan Mj}
\address{\parbox{\linewidth}{School of Mathematics, Tata Institute of Fundamental Research, \\
1, Homi Bhabha Road,  Mumbai-400005, India
}}
\email{mahan@math.tifr.res.in}
\urladdr{https://mathweb.tifr.res.in/~mahan/}

\begin{abstract}
We prove the existence of strongly converging unitary representations in various new settings of countable groups, in particular, arising as extensions by exact groups: semidirect products with amenable groups; generalized wreath products with abelian base; graph wreath products; various free-by-cyclic groups including Gersten's group; and general Bernoulli shift crossed products on $C^*$-algebras. 

\end{abstract}
\maketitle

\section{Introduction}

In the work \cite{HaagerupThorbjornsen2005}, Haagerup and Thorbjornsen strengthened Voiculescu's asymptotic freeness theorem \cite{Voiculescu1991} into the regime of \emph{strong convergence}. In particular, their result implies that the nonabelian free group $\mathbb{F}_2= \langle a, b\rangle$ admits a sequence of unitary representations $\rho_n: G\to \mathcal{U}(\mathbb{M}_n(\mathbb{C}))$ such that for any noncommutative $*$-polynomial $P$ on two variables, $\|\rho_n(P(a,b)) \|\to \|P(a,b)\|$ where the norm on the right hand side is in the left regular representation. Despite originally having $C^*$-algebraic motivations \cite{voiculescu1993around}, strong convergence as a program of research has multiple interdisciplinary applications, including but not limited to von Neumann algebras \cite{HayesPT}, spectral graph theory \cite{BordenaveCollins2019}, optimal spectral gaps \cite{HideMagee2023}, minimal surface theory \cite{song2025randomharmonicmapsspheres}. We refer the reader to the recent surveys \cite{vanhandel2025strongconvergenceshortsurvey, magee2025strong} for more references and applications.

Concerning strong convergence properties for countable groups, we isolate three existential notions that we concern ourselves with in this paper: \begin{enumerate}
    \item $G$ is said to be MF (\emph{matricial field}) \cite{BlackadarKirchberg1997} if there exists a sequence of approximate homomorphisms $\rho_n: G\to \mathcal{U}(\mathbb{M}_n(\mathbb{C}))$ that strongly converge to the left regular representation. 
    \item $G$ is said to be PMF (\emph{purely matricial field}) \cite{Magee_2024} if there exists a sequence of homomorphisms $\rho_n: G\to \mathcal{U}(\mathbb{M}_n(\mathbb{C}))$ that strongly converge to the left regular representation. 
    \item $G$ is said to be PFF (\emph{purely finite field}) \cite{gao2026newsource} if there exists a sequence of homomorphisms $\rho_n: G\to \mathcal{U}(\mathbb{M}_n(\mathbb{C}))$ that strongly converge to the left regular representation, such that $\rho_n(G)$ is finite for each $n\in \mathbb{N}$.
    \end{enumerate}

Recent years have witnessed a substantial study into these properties \cite{BordenaveCollins2019, loudermagee2025limitgroups, MageeThomas2023, Magee_2024, magee2025strong, schafhauser2024finite, RainoneSchafhauser2019, collins2014strong}. In a recent development \cite{gao2026newsource}, a new $C^*$-algebraic approach was introduced and used to prove these properties in vast new families of countable groups. Most notably, \cite{gao2026newsource} proved the PFF property in broad generality: virtually special groups such as fundamental groups of closed hyperbolic 3-manifolds, free products, group doubles, graph products, and more general amalgamated free products. 

In the present article we expand the family of countable groups admitting strongly converging unitary representations in a different direction. We focus here on \emph{extensions} by exact groups. Before we state our results, we first remark that exactness is a natural assumption. Indeed, even in the case of trivial actions, i.e, direct products, preservation of strong convergence is unknown in the absence of this assumption. Our first main result extends the MF phenomena proved in \cite[Theorem 1.4]{RainoneSchafhauser2019}.

\begin{thm}\label{intro thm 1}
    Let $G$ be a finitely generated residually finite amenable group and $L$ be an exact MF/PMF/PFF group. Then $G\rtimes L$ is MF/PMF/PFF.
\end{thm}

A particular family of crossed products that have not been addressed before are wreath products. We are able to give a somewhat definitive result here when the base is abelian. In fact our result address the general setting of generalized wreath products. Soficity for such groups in great generality has been addressed in \cite{gao2025soficity}. Before we state the result, recall that a subgroup $H<G$ is \emph{separable} if $H=\cap_{i\in \mathbb{N}}H_i$ where $H_i<H$ is finite index for each $i\in \mathbb{N}$.  

\begin{thm}\label{intro thm 2}
    Let $G$ be a residually finite abelian group and $L$ be a residually finite exact PMF/PFF group. Then $G\wr L$ is PMF/PFF. More generally, if $G$ is residually finite abelian, $L$ is exact PMF/PFF, and $L\curvearrowright I$ is any transitive action with separable stabilizers, then the generalized wreath product $\oplus_{I}G\rtimes L$ is PMF/PFF. 
\end{thm}

Our next two results address  specific families of free-by-cyclic groups.

\begin{thm}\label{intro thm 3}
    $\F_n \rtimes \Z$ is PFF, where $\F_n = \F_k \ast \F_{n - k}$ and $\Z$ acts by multiplying the standard free generators of $\F_{n - k}$ on the left and/or right by powers of conjugates (in $\F_k$) of the standard free generators of $\F_k$.
\end{thm}

Note that the Gersten's group $\F_3 \rtimes \Z$, where $\F_3 = \langle a, b, c \rangle$ and $\Z$ acts by $a \mapsto a$, $b \mapsto ba$, $c \mapsto ca^2$, satisfies the condition of the above Theorem and is thus PFF. Interestingly Gersten's group is known to not be virtually special, and therefore is a geniunely new example not covered by existing results. The above result also accommodates the following example (Corollary \ref{col: one gen of braid}): $\F_3 \rtimes \Z$, where $\F_3 = \langle a, b, c \rangle$ and $\Z$ acts by $a \mapsto a$, $b \mapsto b^{-1}c^{-1}bcb$, $c \mapsto b^{-1}cb$. This is of minor relevance to the pure braid group on four strands. It is known to experts that this group is isomorphic to $\mathbb{Z}\times (\mathbb{F}_3\rtimes \mathbb{F}_2)$, where the action of one of the generators of $\mathbb{F}_2$ corresponds to the example above. We thank Jingyin Huang for explaining this to us. 

 We also remark that the above result applies to  $\F_2 \rtimes \Z$, where $\Z$ acts via any conjugate (in $\GL(2, \Z)$) of either $\begin{bmatrix}
            1 & 1 \\ 0 & 1
        \end{bmatrix}$ or $\begin{bmatrix}
            -1 & 1 \\ 0 & -1
        \end{bmatrix}$. Based on the structural results of \cite{BestvinaFeighnHandel2005Tits2}, it is plausible that all virtually free groups whose $\mathbb{Z}$-actions induce polynomially growing automorphisms are PFF. We are unable to settle this at present, but are able to prove the following result towards this.

\begin{thm}\label{intro thm 4}
    $\F_n \rtimes \Z$ is PMF, where $\F_n = \F_k \ast \F_{n - k}$ and $\Z$ acts by multiplying the standard free generators of $\F_{n - k}$ on the left and/or right by elements of $\F_k$.
\end{thm}

The upcoming two results are only concerning the MF property. Firstly, we prove MF for general Bernoulli shift crossed products of exact MF groups on MF $C^*$-algebras, with an additional technical assumption (either exactness or an assumption that the infinite tensor product is MF). Notably, this result is able to move significantly beyond amenability requirements on the base algebra (\cite{RainoneSchafhauser2019}). However, this comes at the cost of not being able to prove PMF/PFF in this setting. 

\begin{thm}\label{intro thm 5}
    Let $(A, \varphi)$ be an MF unital exact $C^*$-probability space (or more generally a unital $C^*$-probability space $(A, \varphi)$ s.t. $(A^{\otimes \infty}, \varphi^{\otimes \infty})$ is MF), $L$ be an exact MF group, and $L\curvearrowright I$ be any transitive action with separable stabilizers. Then the generalized Bernoulli shift crossed product $(\otimes_I A)\rtimes L$ is MF. The same holds for the generalized free Bernoulli shift crossed product $(\ast_I A)\rtimes L$, in which case we do not need to assume $A$ is exact.
\end{thm}

Our next MF result concerns graph wreath products. This is a family of groups that simultaneously generalize both graph products and wreath products. More precisely, let $\Gamma$ be a possibly infinite simplicial graph, and let $G \curvearrowright \Gamma$ be a graph action of a countable group $G$. Connes embeddability and soficity of such groups in great generality was proved in \cite{gao2024soficactionsgraphs}. In order to state our next main result, we need to define the following notion. An action of a countable group $G$ on a simplicial graph $\Gamma$ is said to be \emph{residually finite} if it is sofic in the sense of \cite[Definition 1.1]{gao2024soficactionsgraphs}, where $\epsilon = 0$ and the map $\varphi$ is required to be an actual group homomorphism (see Definition \ref{defn: res fin action} for details). Several natural examples of these actions exist (see Example \ref{exs: res fin action}).

\begin{thm}\label{intro thm 6}
    Let $G$ be a residually finite exact MF group, $H$ be an exact MF group, $\Gamma = (V, E)$ be a (possibly infinite) simplicial graph, and $H\curvearrowright\Gamma$ be a residually finite action. Then the graph wreath product $\star_{\Gamma} G\rtimes H$ is MF.
\end{thm}

Now we describe some insights concerning the proofs of our results. The recent works \cite{gao2026toeplitzexactnessstrongconvergence, gao2026newsource}, emphasized a general strategy of proving the existence of strongly converging unitary representations involving two stages. The first stage is to set up \emph{ambient strong convergence}. The second stage involves an \emph{upgrading procedure} to obtain the full strong convergence of matrix models. Ambient strong convergence is typically set up for one of the building blocks of the groups in question using existing strongly convergent models. The precise role of ambient strong convergence is to be able to locate the entire reduced $C^*$-algebra of the building block (not just the group ring) inside a large norm--ultraproduct $C^*$-algebra via a sequence of desirable maps. This then allows one to use $C^*$-algebraic techniques to design upgrading procedures, to extend the embedding to the entire reduced $C^*$-algebra of the group in question. In \cite{gao2026newsource}, the upgrading procedure was carried out using the remarkable work of Pimsner \cite{pimsner1997class}. More generally, in \cite{gao2026toeplitzexactnessstrongconvergence}, a general machine called ``Toeplitz exactness" was proved, to upgrade strong convergence into Toeplitz--Pimnser $C^*$-algebras associated to arbitrary $C^*$-correspondences via the gauge--invariant uniqueness theorem. In the present paper, the upgrading is achieved via a rather direct and succinct argument inspired by a key idea in \cite{RainoneSchafhauser2019}, circumventing the above deep machinery. We are able to carry out a rather elementary ``double crossed product'' trick (Proposition \ref{prop: strong convergence upgrade}) combined with Fell's absorption and exactness to carry out upgrading for extensions. Proving the necessary ambient strong convergence is the rest of the puzzle, wherein we have to design such models in each of the cases addressed in our results, often using the specific algebraic and combinatorial information. 

We take this opportunity to formulate two natural open questions. Firstly, it is natural to now ask if all free-by-cyclic groups are PFF. More generally, what about free-by-free groups? A specific instance of this is whether the pure braid group on four strands is PFF.

\subsection*{Acknowledgements} We are indebted to the UMD Department of Mathematics and the Brin Mathematics Research Center, which supported the joint visit of the first and third authors to UMD in April 2026, wherein this project was completed. The second author is grateful to R. van Handel for several insightful conversations and constant encouragement. 

\subsection*{AI statement} No AI tools were used in the research process and preparation of this manuscript. 

\section{Proof of main results}

\subsection{Upgrading strong convergence for extensions}
\begin{defn}[Ambient strong convergence]
    Fix an index set $I$ and a discrete group $G$. Let $A\sk$ be a $C^\ast$-algebra, $\alpha\sk: G \curvearrowright A\sk$ be an action, and $(X_i\sk)_{i \in I}$ be a tuple in $A\sk$, for each $k \in \N \cup \{\infty\}$. Assume further that $A\si$ is generated by $(X_i\si)_{i \in I}$. We say $(X_i\sk; \alpha\sk)_{i \in I}$ \emph{converges strongly} to $(X_i\si; \alpha\si)_{i \in I}$ if,
    \begin{enumerate}
        \item $(X_i\sk)_{i \in I} \to (X_i\si)_{i \in I}$ strongly in the classical sense;
        \item for any $\ast$-polynomials $P, Q$ and $g \in G$,
        \begin{equation*}
            \hspace{0.8 cm}\lim_{k \to \infty} \|\alpha_g\sk(P(X_i\sk)) - Q(X_i\sk)\| = \|\alpha_g\si(P(X_i\si)) - Q(X_i\si)\|.
        \end{equation*}
    \end{enumerate}
\end{defn}

\begin{rmk}
    We leave it to the reader to check that the definition above is equivalent to the following: For each free ultrafilter $\cU$ on $\N$, there exists an embedding $\pi: A\si \to \prod_\cU A\sk$ such that
    \begin{enumerate}
        \item $\pi(X_i\si) = (X_i\sk)_\cU$ for all $i \in I$; and,
        \item for $a \in A\si$, $g \in G$, if $\pi(a) = (a\sk)_\cU$, then $\pi(\alpha_g\si(a)) = (\alpha_g\sk(a\sk))_\cU$.
    \end{enumerate}
\end{rmk}

\begin{prop}\label{prop: strong convergence upgrade}
    Fix an index set $I$ and let $G$ be an exact group. Let $A\sk$ be a $C^\ast$-algebra, $\alpha\sk: G \curvearrowright A\sk$ be an action, and $(X_i\sk)_{i \in I}$ be a tuple in $A\sk$, for each $k \in \N \cup \{\infty\}$. Assume further that $A\si$ is generated by $(X_i\si)_{i \in I}$. If $(X_i\sk; \alpha\sk)_{i \in I}$ converges to $(X_i\si; \alpha\si)_{i \in I}$ strongly, then $(X_i\sk, \lambda_g)_{i \in I, g \in G}$ in $A\sk \rtimes_{\alpha\sk} G$ converges strongly to $(X_i\si, \lambda_g)_{i \in I, g \in G}$ in $A\si \rtimes_{\alpha\si} G$.
\end{prop}

\begin{proof}
    Fix a free ultrafilter $\cU$ on $\N$. Let $\pi: A\si \to \prod_\cU A\sk$ be the embedding defined by $\pi(X_i\si) = (X_i\sk)_\cU$ for all $i \in I$. It suffices to show that $\pi$ extends to an embedding $\tilde\pi: A\si \rtimes_{\alpha\si} G \to \prod_\cU A\sk \rtimes_{\alpha\sk} G$ s.t. $\tilde\pi(\lambda_g) = (\lambda_g)_\cU$ for all $g \in G$.

    Now, as $A\sk \subset A\sk \rtimes_{\alpha\sk} G$, we can regard $\pi$ as having codomain $$B = \prod_\cU A\sk \rtimes_{\alpha\sk}G.$$ $G$ acts on $B$ via the inner action $\beta$ defined by
    \begin{equation*}
        G \ni g \mapsto \operatorname{Ad}((\lambda_g)_\cU) \in \operatorname{Aut}(B).
    \end{equation*}

    The ambient strong convergence assumption ensures that $\pi$ is then $G$-equivariant, so $\pi$ extends to an embedding $\pi_1: A\si \rtimes_{\alpha\si} G \to B \rtimes_\beta G$ that acts as the identity map on $G$. Now, as $\beta$ is inner, $B \rtimes_\beta G$ is in fact isomorphic to $B \otimes C_r^\ast(G)$ by the isomorphism that acts as the identity on $B$ and sends
    \begin{equation*}
        B \rtimes_\beta G \ni \lambda_g \mapsto (\lambda_g)_\cU \otimes \lambda_g \in B \otimes C_r^\ast(G).
    \end{equation*}

    Since $G$ is exact,
    \begin{equation*}
        B \otimes C_r^\ast(G) = \left[\prod_\cU A\sk \rtimes_{\alpha\sk} G\right] \otimes C_r^\ast(G) \cong \prod_\cU \left[\left(A\sk \rtimes_{\alpha\sk} G\right) \otimes C_r^\ast(G)\right]
    \end{equation*}

    canonically. Now, the composition of $\pi_1$ with the two isomorphisms above extends $\pi$ and sends $\lambda_g$ to $(\lambda_g \otimes \lambda_g)_\cU$. Hence, the range is contained in the ultraproduct of the ranges of the comultiplication maps $A\sk \rtimes_{\alpha\sk} G \to \left(A\sk \rtimes_{\alpha\sk} G\right) \otimes C_r^\ast(G)$. Applying the inverses of the comultiplication maps yields an embedding of $A\si \rtimes_{\alpha\si} G$ into $B$, and it is straightforward to check that this is indeed the desired $\tilde\pi$.
\end{proof}

\subsection{Proof of Theorem \ref{intro thm 1}}

We recall the following standard group-theoretic fact. We include a proof here for the convenience of the reader.

\begin{prop}
    Let $G$ be a finitely generated residually finite group. Then $G$ admits a decreasing sequence of finite-index characteristic subgroups $H_n$ s.t. $\bigcap_n H_n = \{e\}$.
\end{prop}

\begin{proof}
    We note that $G$, being finitely generated, admits only finitely many transitive actions on $\{1, \cdots, n\}$ as any such action is determined by how the finitely many generators of $G$ act. Such transitive actions are in one-to-one correspondence with subgroups of index $n$. Hence, there are only finitely many subgroups of $G$ of index at most $n$, so $H_n$, the intersection of all such subgroups, is of finite index. It is easy to see $H_n$ satisfies the desired conditions.
\end{proof}

\begin{proof}[Proof of Theorem \ref{intro thm 1}]
    Let $H_n$ be a decreasing sequence of finite-index characteristic subgroups of $G$ s.t. $\bigcap_n H_n = \{e\}$. Then as $G$ is amenable, the $\ast$-homomorphism
    \begin{equation*}
        C_r^\ast(G) \to \prod_\cU C_r^\ast(G/H_n), \lambda_g \mapsto (\lambda_{gH_n})_\cU
    \end{equation*}
    is well-defined. Since it is trace-preserving, it is an embedding. As $H_n$ is characteristic, the action $\alpha$ of $L$ on $G$ descends to an action $\alpha_n$ of $L$ on $G/H_n$. One easily check, that $\alpha_n \to \alpha$ strongly. Hence, by Proposition \ref{prop: strong convergence upgrade}, we obtain an embedding
    \begin{equation*}
        C_r^\ast(G \rtimes L) = C_r^\ast(G) \rtimes L \hookrightarrow \prod_\cU (C_r^\ast(G/H_n) \rtimes L) = \prod_\cU (C_r^\ast(G/H_n \rtimes L)).
    \end{equation*}

    Note that this embedding is induced by a sequence of group homomorphisms $G \rtimes L \to G/H_n \rtimes L$, indeed the natural quotient map. We also have $G/H_n \rtimes L \hookrightarrow (G/H_n \rtimes \Aut(G/H_n)) \times L$ by sending $g \in G/H_n$ to $(g, e)$ and $l \in L$ to $(\alpha_n(l), l)$. Since $G/H_n \rtimes \Aut(G/H_n)$ is finite, the result follows.
\end{proof}

\subsection{Proof of Theorem \ref{intro thm 2}}

\begin{proof}[Proof of Theorem \ref{intro thm 2}]
    We directly prove the ``more generally" statement. Since $L \curvearrowright I$ is transitive, we may write $I = L/H$ for some $H < L$ and the action is simply given by left multiplication. Per our assumption, $H < L$ is separable, so there exists a decreasing sequence $H_n < L$ s.t. $H = \bigcap_n H_n$. Let $\pi_n: L/H \to L/H_n$ be the natural quotient maps. Now, we consider the embedding,
    \begin{equation*}
        C_r^\ast(\oplus_{L/H} G) \hookrightarrow \prod_\cU C_r^\ast(\oplus_{L/H_n} G)
    \end{equation*}
    induced by the sequence of group homomorphisms
    \begin{equation*}
        \bigoplus_{L/H} G \to \bigoplus_{L/H_n} G; (g_x)_{x \in L/H} \mapsto \left(\sum_{x \in L/H, \pi_n(x) = y} g_x\right)_{y \in L/H_n}.
    \end{equation*}
    
    This is well-defined and continuous because $\oplus_{L/H} G$ is abelian and thus amenable. It is an embedding because $\bigcap_n H_n = H$ so the sequence of homomorphisms is eventually separating and thus weakly converging. Thus, per Proposition \ref{prop: strong convergence upgrade},
    \begin{equation*}
        C_r^\ast(\oplus_{L/H} G \rtimes L) = C_r^\ast(\oplus_{L/H} G) \rtimes L \hookrightarrow \prod_\cU (C_r^\ast(\oplus_{L/H_n} G) \rtimes L) = \prod_\cU C_r^\ast(\oplus_{L/H_n} G \rtimes L)
    \end{equation*}
    where $L$ acts on $\oplus_{L/H_n} G$ by permuting the copies of $G$ via the left multiplication action $L \curvearrowright L/H_n$. Note that this embedding is induced by a sequence of group homomorphisms $\oplus_{L/H} G \rtimes L \to \oplus_{L/H_n} G \rtimes L$. We also have the action of $L$ on $\oplus_{L/H_n} G$ quotients through the finite group $\Sym(L/H_n)$, and we shall denote $\alpha_n: L \to \Sym(L/H_n)$. So,
    \begin{equation*}
        \oplus_{L/H_n} G \rtimes L \hookrightarrow (\oplus_{L/H_n} G \rtimes \Sym(L/H_n)) \times L
    \end{equation*}
    by acting as the identity on $\oplus_{L/H_n} G$ and sending $l \in L$ to $(\alpha_n(l), l)$. Since $G$ is abelian and residually finite, so is $\oplus_{L/H_n} G$ and thus $\oplus_{L/H_n} G$ is PFF. Hence, $\oplus_{L/H_n} G \rtimes \Sym(L/H_n)$ is a finite-index extension of a PFF group and thus PFF itself. The result follows.
\end{proof}

\subsection{Proof of Theorem \ref{intro thm 3}}

\begin{proof}[Proof of Theorem \ref{intro thm 3}]
    By free exactness \cite{skoufranis2015notion} (see also \cite{PisierRDP} and \cite{gao2026toeplitzexactnessstrongconvergence}), we have the embedding
    \begin{equation*}
        C_r^\ast(\F_n) \hookrightarrow \prod_{m \to \cU} C_r^\ast((\Z/m\Z)^k \ast \F_{n - k})
    \end{equation*}
    induced by the sequence of group homomorphisms sending each of the first $k$ standard free generators of $\F_n$ to the standard generator of the corresponding copy of $\Z/m\Z$. The action of $\Z$ on $\F_n$ descends to actions of $\Z$ on $(\Z/m\Z)^k \ast \F_{n - k}$, by multiplying by the same elements in $\F_k$ except quotienting down to $(\Z/m\Z)^k$. Since all the elements multiplied are powers of conjugates of the free generators, in $(\Z/m\Z)^k \ast \F_{n - k}$ they all have orders divisible by $m$. Thus, the action of $\Z$ quotients through $\Z/m\Z$, so by Proposition \ref{prop: strong convergence upgrade},
    \begin{equation*}
    \begin{split}
        C_r^\ast(\F_n \rtimes \Z) &\hookrightarrow \prod_{m \to \cU} C_r^\ast([(\Z/m\Z)^k \ast \F_{n - k}] \rtimes \Z)\\
        &\hookrightarrow \prod_{m \to \cU} C_r^\ast(([(\Z/m\Z)^k \ast \F_{n - k}] \rtimes \Z/m\Z) \times \Z).
    \end{split}
    \end{equation*}

    Note that the RHS groups are
    direct products of virtually free groups and $\Z$, and so are PFF. Since the embedding is induced by group homomorphisms, the result follows.
\end{proof}

As a corollary of the above, we also obtain that,

\begin{col}\label{col: one gen of braid}
    $\F_3 \rtimes \Z$ is PFF, where $\F_3 = \langle a, b, c \rangle$ and $\Z$ acts by $a \mapsto a$, $b \mapsto b^{-1}c^{-1}bcb$, $c \mapsto b^{-1}cb$.
\end{col}

\begin{proof}
    Re-choose the generators of $\F_3$ by defining,
    \begin{equation*}
        x = a, \; y = cb, \; z = ba.
    \end{equation*}

    Then one may directly verify that $\Z$ acts by
    \begin{equation*}
        x \mapsto x, y \mapsto y, z \mapsto y^{-1}zx^{-1}yx.
    \end{equation*}

    The result now follows from Theorem \ref{intro thm 3}.
\end{proof}

\subsection{Proof of Theorem \ref{intro thm 4}}

\begin{proof}[Proof of Theorem \ref{intro thm 4}]
    We start by using the fact that $\F_k$ is PFF \cite{BordenaveCollins2019}, i.e., there is a sequence of representations $\F_k \to U(m)$ having finite ranges that strongly converges to the left regular representation. To put it another way
    \begin{equation*}
        C_r^\ast(\F_k) \hookrightarrow \prod_\cU \M_m(\C)
    \end{equation*}
    induced by group homomorphisms $\F_k \to U(m)$ with finite ranges. By free exactness,
    \begin{equation*}
        C_r^\ast(\F_n) = C_r^\ast(\F_k \ast \F_{n - k}) \hookrightarrow \prod_{m \to \cU} \M_m(\C) \ast C_r^\ast(\F_{n - k}).
    \end{equation*}

    The action of $\Z$ on $C_r^\ast(\F_n)$ descends down to actions on $\M_m(\C) \ast C_r^\ast(\F_{n - k})$ by multiplying the standard free generators of $\F_{n - k}$ by the unitaries in $\M_m(\C)$ corresponding to the words in $\F_k$ that are multiplied in the original action. Since the group homomorphisms $\F_k \to U(m)$ have finite ranges, all these generators have finite orders, so the action of $\Z$ has some finite order $r(m)$ as well. Whence, by Proposition \ref{prop: strong convergence upgrade},
    \begin{equation*}
    \begin{split}
        C_r^\ast(\F_n \rtimes \Z) &\hookrightarrow \prod_{m \to \cU} [\M_m(\C) \ast C_r^\ast(\F_{n - k})] \rtimes \Z\\
        &\hookrightarrow \prod_{m \to \cU} ([\M_m(\C) \ast C_r^\ast(\F_{n - k})] \rtimes [\Z/r(m)\Z]) \otimes C_r^\ast(\Z).
    \end{split}
    \end{equation*}

    The result follows.
\end{proof}

\subsection{Proof of Theorem \ref{intro thm 5}}

\begin{proof}[Proof of Theorem \ref{intro thm 5}]
    Since $L \curvearrowright I$ is transitive, we may write $I = L/H$ for some $H < L$ and the action is simply given by left multiplication. Per our assumption, $H < L$ is separable, so there exists a decreasing sequence $H_n < L$ s.t. $H = \bigcap_n H_n$. Let $\pi_n: L/H \to L/H_n$ be the natural quotient maps. Now, we fix an increasing sequence of finite subsets $F_k \subset L/H$ s.t. $\bigcup_k F_k = L/H$. Since $H = \bigcap_n H_n$, we may choose an increasing sequence $n_k$ s.t. $\pi_{n_k}$ is injective when restricted to $F_k$. Let $E_k: \otimes_{L/H} A \to \otimes_{F_n} A$ be the expectations. Let $\iota_k: \otimes_{F_n} A \to \otimes_{L/H_{n_k}} A$ be the embeddings induced by $\pi_{n_k}|_{F_k}$. Then we have
    \begin{equation*}
        \bigotimes_{L/H} A \to \prod_\cU \left(\bigotimes_{L/H_{n_k}} A\right)
    \end{equation*}
    induced by $\iota_k \circ E_k$. Since $\bigcup_k F_k = L/H$, it is easy to see that the map is multiplicative and in fact an embedding. Let $\alpha: L \curvearrowright \otimes_{L/H} A$ be the action induced by left multiplication $L \curvearrowright L/H$ and similarly for $\alpha_n: L \curvearrowright \otimes_{L/H_n} A$. Then it is easy to check $\alpha_n \to \alpha$ strongly. Per Proposition \ref{prop: strong convergence upgrade},
    \begin{equation*}
        \left(\bigotimes_{L/H} A\right) \rtimes L \hookrightarrow \prod_\cU \left(\bigotimes_{L/H_{n_k}} A\right) \rtimes L.
    \end{equation*}

    It is standard that
    \begin{equation*}
        \left(\bigotimes_{L/H_{n_k}} A\right) \rtimes L \hookrightarrow \left[\left(\bigotimes_{L/H_{n_k}} A\right) \rtimes \Sym(L/H_{n_k})\right] \otimes C_r^\ast(L)
    \end{equation*}
    by sending $\lambda_l \to \lambda_{l'} \otimes \lambda_l$ where $l' \in \Sym(L/H_{n_k})$ is the permutation given by left multiplication by $l$. Since $\Sym(L/H_{n_k})$ is finite, the RHS is MF, so the result follows. The free version can be proved analogously, where we instead use free exactness to show free products of MF algebras are MF.
\end{proof}

\subsection{Proof of Theorem \ref{intro thm 6}}

\begin{defn}\label{defn: res fin action}
    Let $\alpha: G \curvearrowright \Gamma$ be an action of a countable group on a simplicial graph $\Gamma = (V, E)$. We say $\alpha$ is \emph{residually finite} if, for each finite subset $F \subset G$ and $W \subset V$, there exists a finite graph $\Theta$, a finite set $A$, a group homomorphism $\varphi: G \to \Sym(A)$, and, for each $a \in A$, a graph embedding
    \begin{equation*}
        \pi_a: (W, E|_{W \times W}) \hookrightarrow \Theta
    \end{equation*}
    s.t. $\pi_{\varphi(g)a}(v) = \pi_a(\alpha(g^{-1})v)$ for all $a \in A$, $g \in F$, $v \in W$, whenever $\alpha(g^{-1})v \in W$.
\end{defn}

\begin{rmk}
    Note that the above definition is a strengthening of the definition of sofic actions on graphs in \cite{gao2024soficactionsgraphs}, by requiring $\epsilon = 0$ and $\varphi$ to be an actual group homomorphism. In case $\Gamma$ is either complete or totally disconnected, the definition is a strengthening of sofic actions on sets (namely, the vertex set) in \cite{gao2025soficity}, again by requiring $\epsilon = 0$ and $\varphi$ to be an actual group homomorphism.
\end{rmk}

\begin{examples}\label{exs: res fin action}
    The following actions are residually finite:
    \begin{enumerate}
        \item For the purposes of this example, a Cayley graph $\Gamma=\Gamma(G,S)$ of a group $G$ with respect to a symmetric set $S$ has vertex set 
        given by elements of $g$, with edges corresponding to unordered pairs $\{g,h\}$ where $g^{-1}h \in S$.
        Note that $S$ is allowed to be infinite. More drastically, we allow for the possibility that $S$ \emph{does not} generate $G$, so that $\Gamma$ is allowed to be disconnected.
        
        The action of a residually finite group $G$ on any of its Cayley graphs in the above sense is residually finite. This follows from picking $\Theta$ to be some appropriately chosen Cayley graph of a sufficiently large finite quotient group $G/H$, $A = G/H$, and $\pi_a$ to be the composition of the quotient map $G \to G/H$ with left-multiplication by $a^{-1}H$. In fact, $G$ being residually finite is equivalent to the action of $G$ on some, equivalently all, its Cayley graphs being residually finite. We leave the reader to check the details;
        \item The left multiplication action of $G$ on $G/H$, equipped with either the complete graph or the totally disconnected graph structure, where $H < G$ is separable. This follows from Proposition 5 and the remark preceding Proposition 8 in \cite{Gao2025};
        \item Any action of a LERF group $G$ on any complete or totally disconnected graph \cite[Proposition 8]{Gao2025};
        \item Any action of a free group on any graph. This follows from the proof of \cite[Theorem 2.14]{gao2024soficactionsgraphs}.
    \end{enumerate}
\end{examples}

\begin{proof}[Proof of Theorem \ref{intro thm 6}]
    Fix increasing sequences of finite subsets $F_n \subset H$ and $W_n \subset V$ s.t. $\bigcup_n F_n = H$ and $\bigcup_n W_n = V$. For each $n$, fix a finite graph $\Theta_n$, a finite set $A_n$, a group homomorphism $\varphi_n: H \to \Sym(A_n)$, and, for each $a \in A$, a graph embedding $\pi_{n, a}: (W_n, E|_{W_n \times W_n}) \hookrightarrow \Theta_n$ as in the definition of residually finite actions.

    We have the following natural embeddings,

    \begin{equation*}
        C_r^\ast(\star_\Gamma G) \hookrightarrow \prod_\cU C_r^\ast(\star_{(W_n, E|_{W_n \times W_n})} G) \hookrightarrow \prod_\cU \M_{|A_n|}(C_r^\ast(\star_{\Theta_n} G))
    \end{equation*}
    where the first embedding is by noting that $\star_{(W_n, E|_{W_n \times W_n})} G$ is an increasing sequence of subgroups of $\star_\Gamma G$ whose union is the entire group, so the embedding may be defined by taking the sequence of expectations associated to these subgroups; and where the second embedding is defined by noting that, for each $a \in A_n$, $\pi_{n, a}$ induces a group embedding $\star_{(W_n, E|_{W_n \times W_n})} G \hookrightarrow \star_{\Theta_n} G$ and thus a $C^\ast$-algebraic embedding $C_r^\ast(\star_{(W_n, E|_{W_n \times W_n})} G) \hookrightarrow C_r^\ast(\star_{\Theta_n} G)$, and one then uses these embeddings to send an element of $C_r^\ast(\star_{(W_n, E|_{W_n \times W_n})} G)$ to a diagonal matrix with entries in $C_r^\ast(\star_{\Theta_n} G)$.

    Note that $\varphi_n: H \to \Sym(A_n)$ induces an action of $H$ on $\M_{|A_n|}(C_r^\ast(\star_{\Theta_n} G))$ by going through permutation matrices. The definition of residually finite actions ensures these actions strongly converge to the action of $H$ on $C_r^\ast(\star_\Gamma G)$. Thus, by \ref{prop: strong convergence upgrade}, we have the embedding,

    \begin{equation*}
        C_r^\ast(\star_\Gamma G \rtimes H) = C_r^\ast(\star_\Gamma G) \rtimes H \hookrightarrow \prod_\cU \M_{|A_n|}(C_r^\ast(\star_{\Theta_n} G)) \rtimes H.
    \end{equation*}

    It is standard that
    \begin{equation*}
        (\M_{|A_n|}C_r^\ast(\star_{\Theta_n} G)) \rtimes H \hookrightarrow \left[\M_{|A_n|}(C_r^\ast(\star_{\Theta_n} G)) \rtimes \Sym(A_n)\right] \otimes C_r^\ast(H)
    \end{equation*}
    by sending $\lambda_h \to \lambda_{\varphi_n(h)} \otimes \lambda_h$. By \cite[Corollary 1.3]{gao2026newsource}, $\star_{\Theta_n} G$ is MF and thus $\M_{|A_n|}(C_r^\ast(\star_{\Theta_n} G)) \rtimes \Sym(A_n)$ is MF as $A_n$ is finite. The result follows.
\end{proof}

\bibliographystyle{amsalpha}
\bibliography{bibliography}

\end{document}